\documentclass[12pt,a4paper]{amsart}
\usepackage{dsfont}
\usepackage{graphicx}
\usepackage{amssymb}
\usepackage[all]{xy}
\usepackage{amsmath}
\usepackage{tikz}
\usepackage[french,english]{babel}
\usepackage{amssymb,amsmath,amsfonts,amsthm,amscd}
\usepackage{indentfirst,graphicx}
\usepackage[font={scriptsize},captionskip=5pt]{subfig}
\usepackage[T1]{fontenc}

\newtheorem{theorem}{Theorem}[section]
\newtheorem{corollary}[theorem]{Corollary}
\newtheorem{proposition}[theorem]{Proposition}
\newtheorem{lemma}[theorem]{Lemma}

\newtheorem{definition}[theorem]{Definition}

\newtheorem{remark}[theorem]{Remark}

 \let\:=\colon  
   \let\?=\overline

\def\and{\hbox{ and }}

\let\:=\colon  
   \let\?=\overline

\let\Sum=\sum \def\sum{\Sum\nolimits}

\def\and{\hbox{ and }}

\def\DONE{*!*}
\def\NextDef #1 {\def\NextOne{#1}
 \ifx\NextOne\DONE\let\next\relax
 \else\expandafter\xdef\csname#1\endcsname{\TheOp}
  \let\next\NextDef
 \fi \next}
\def\TheOp{\mathop{\rm\NextOne}}
 \NextDef 
  Projan Supp Proj Sym Spec Hom cod Ker dist
 *!*

\newcommand{\cC}{{\mathcal C}}
\newcommand{\cD}{{\mathcal D}}

\newcommand{\cO}{{\mathcal O}}

\newcommand{\bC}{{\mathbb{C}}}

\newcommand{\bN}{{\mathbb{N}}}
\def\TheOp{\hbox{\rm\NextOne}}
\NextDef 
 A ICIS 
 *!*

\begin{document}

\title{CATEGORICAL ASPECTS OF THE DOUBLE STRUCTURE OF A MODULE}
\author{Thiago F. da Silva}

\maketitle

\begin{abstract}
{\small In this work we develop some categorical aspects of the double structure of a module, defined in \cite{GT}.}
\end{abstract}

\let\thefootnote\relax\footnote{2010 \textit{Mathematics Subjects Classification} 18B99, 14J17, 13C60
	
	}

\section*{Introduction}

The study of bi-Lipschitz equisingularity was started at the end of 1960's with works of Zariski \cite{Za}, Pham \cite{Pham} and Teissier \cite{PT}. At the end of 1980's, Mostowski \cite{M1} introduced a new technique for the study of Bi-Lipschitz equisingularity from the existence of Lipschitz vector fields.

In \cite{G1} Gaffney defined the concept of the double of an ideal and developed the infinitesimal Lipschitz conditions for a family of hypersurfaces using the integral closure of modules, namely, the double of some jacobian ideals. In \cite {G2} Gaffney used the double and the integral closure of modules to get algebraic conditions for bi-Lipschitz equisingularity of a family of irreducible curves. In \cite{SGP} the authors also used the double and the integral closure of ideals to get an algebraic condition in order to get a canonical vector field defined along a Essentially Isolated Determinantal Singularities (EIDS) family, which is Lipschitz provided the matrix of deformation of the 1-unfolding which defines the EIDS is constant.

In \cite{GT} it was extended the notion of the double for modules, and we have generalizations for some results of \cite{G1}, further.

In this work, our main goal is to look the categorical properties of the double structure, under an algebraic viewpoint. In \cite{GT} we develop algebraic conditions, in order to ensure the existence of Lipschitz canonical vector fields, using the double and the integral closure of modules.

In section 1 we define the double homomorphism and we get several results that relates standard properties of a homomorphism and its double.

In section 2 we develop some relations between the homological behavior of chain complexes and its doubles.

In section 3 we see the notion of double in a categorical viewpoint, and we prove that the double category is isomorphic to the categories of the modules that are embedded in a free module with finite rank.

In section 4 we find a quite natural way to define the double in a quotient of a free module with finite rank.

Finally, in section 5 we extend the notion of a double homomorphism between two submodules embedded on finite powers of local rings of possibly different analytic varieties, linked by an analytic map-germ between them.

\section*{Acknowledgements}

The author is grateful to Terence Gaffney for the inspiration and support for this work and to Nivaldo Grulha for his careful reading and valuable suggestions.

The author was supported by Funda\c{c}\~ao de Amparo \`a Pesquisa do Estado de S\~ao Paulo - FAPESP, Brazil, grant 2013/22411-2.

\section{THE DOUBLE HOMOMORPHISM AND BASIC PROPERTIES}\label{sec1}

Let $R$ be a ring. 

\begin{definition}
Let $\mathcal T(R)$ be the category of the $R-$modules $M$ which are $R-$submodules of $R^p$, for some natural number $p$.
\end{definition}

Let $X\subset\bC^n$ be an analytic space and let $\cO_X$ be the analytic sheaf of local rings over $X$, and let $x\in X$.

Here we work on the categories $\mathcal T(\cO_{X,x})$ and $\mathcal T(\cO_{X\times X,(x,x)})$.

It is defined in \cite{GT} the concept of a double of a $\cO_{X,x}-$submodule $M$ of $\cO_{X,x}^p$. We recall the definition now.

Consider the projection maps $\pi_1,\pi_2: X\times X\rightarrow X$.

\begin{definition}
\begin{enumerate}

\item Let $h\in \cO_{X,x}^p$. The {\bf double of $h$} is defined as $$h_D:=(h\circ\pi_1,h\circ\pi_2)\in\cO_{X\times X,(x,x)}^{2p}.$$

\item The {\bf double of $M$} is denoted by $M_D$ and is defined as the $\cO_{X\times X,(x,x)}-$submodule of $\cO_{X\times X,(x,x)}^{2p}$ generated by $\{h_D\mbox{ / }h\in M\}$. 

\end{enumerate}
\end{definition}

\vspace{0,5cm}

The first result is a quite useful tool many times when we work with the double.

\begin{proposition}\label{3.P1}
Let $M,N\subset\cO_{X,x}^p$ submodules and $h,g\in\cO_{X,x}^p$. Then:
\begin{enumerate}

\item[a)] h=g if, and only if, $h_D=g_D$;

\vspace{0,2cm}

\item[b)] $h\in M$ if, and only if, $h_D\in M_D$;

\vspace{0,2cm}

\item[c)] $M\subset N$ if, and only if, $M_D\subset N_D$;

\vspace{0,2cm}

\item[d)] $M=N$ if, and only if, $M_D=N_D$.

\end{enumerate}
\end{proposition}

\begin{proof}
(a) The implication ($\Longrightarrow$) is obvious. Suppose that $h_D=g_D$. In particular, $h\circ\pi_1=g\circ\pi_1$, and for all $z$ in a neighborhood of $x$ we have $h(z)=h\circ\pi_1(z,x)=g\circ\pi_1(z,x)=g(z)$, hence $h=g$.

\vspace{0,5cm}

(b) The implication ($\Longrightarrow$) is obvious. Suppose now that $h_D\in M_D$. Then, we can write $$h_D=\sum\alpha_i(g_i)_D$$ with $g_i\in M$ and $\alpha_i\in\cO_{X\times X,(x,x)}$. In particular, $h\circ\pi_1=\sum\alpha_i(g_i\circ\pi_1)$. Taking $\alpha_i^{x}\in\cO_{X,x}$ given by $\alpha_i^{x}(z):=\alpha_i(z,x)$, $\forall i$, we get $h=\sum\alpha_i^{x}g_i$ which belongs to $M$.

\vspace{0,5cm}

(c) The implication ($\Longrightarrow$) is obvious. Suppose that $M_D\subset N_D$. Let $h\in M$ arbitrary. Then $h_D\in M_D\subset N_D$, so by the item (b) we conclude that $h\in N$. Therefore, $M\subset N$.

\vspace{0,5cm}

(d) It is a straightforward consequence of the item (c).

\end{proof}

\begin{corollary}\label{3.C2}
For each $\cO_{X,x}-$submodule $M$ of $\cO_{X,x}^p$, the natural map
$$
\begin{matrix}
D_M:  & M & \longrightarrow & M_D \\
          &  h  & \longmapsto    & h_D
\end{matrix}$$ is an injective group homomorphism. In particular, we can see $M$ as an additive subgroup of $M_D$.
\end{corollary}
\begin{proof}
It is a straightforward consequence of the definition of the double that $D_M$ is a group homomorphism. The Proposition \ref{3.P1} (a) gives the injectivity.
\end{proof}

Our main goal is to give a categorical sense for the double structure. The next theorem is the key for it.

\begin{theorem}\label{3.T3}
Let $M\subset\cO_{X,x}^{p}$ and $N\subset\cO_{X,x}^q$ be $\cO_{X,x}-$submodules. If $\phi: M\rightarrow N$ is an $\cO_{X,x}-$module homomorphism then there exists a unique $\cO_{X\times X,(x,x)}-$module homomorphism $\phi_D: M_D\rightarrow N_D$ such that $\phi_D(h_D)=(\phi(h))_D$, $\forall h\in M$, i.e, the following diagram commutes:

$$\begin{matrix}
       &                           &  \phi                  &                           &           \\
       & M                       &  \longrightarrow  & N                  &            \\
D_M & \downarrow &                           & \downarrow & D_N \\
       & M_D                   &  \dashrightarrow  & N_D                   &              \\     
       &                           &  \phi_D                  &                           &                                        
\end{matrix}$$

The map $\phi_D$ is called the double of $\phi$.
\end{theorem}

\begin{proof}
Since $M_D$ is generated by $\{h_D \mbox{ / }h\in M\}$ then we can define $\phi_D: M_D\rightarrow N_D$ in a quite natural way: for each $u=\sum\limits_i\alpha_i(h_i)_D$ with $\alpha_i\in\cO_{X\times X,(x,x)}$ and $h_i\in M$ we define $$\phi_D(u):=\sum\limits_i\alpha_i(\phi(h_i))_D$$ which belongs to $N_D$.

{\bf Claim : $\phi_D$ is well defined.}
In fact, suppose that $\sum\limits_i\alpha_i(h_i)_D=\sum\limits_j\beta_j(g_j)_D$, with $\alpha_i,\beta_j\in\cO_{X\times X,(x,x)}$ and $h_i,g_j\in M$. So, we get two equations:

$$\sum\limits_i\alpha_i(h_i\circ\pi_1)=\sum\limits_j\beta_j(g_j\circ\pi_1) \eqno (1)$$

$$\sum\limits_i\alpha_i(h_i\circ\pi_2)=\sum\limits_j\beta_j(g_j\circ\pi_2). \eqno (2)$$

Take $U$ an open neighborhood of $x$ in $X$ where $\alpha_i,\beta_j$ are defined on $U\times U$, and $h_i,g_j$ are defined on $U$. For each $w\in U$ define $\alpha_i^w,\beta_j^w\in\cO_{X,x}$ given by the germs of the maps 
$$\begin{matrix}
\alpha_i^w: & U  &  \longrightarrow  &  \bC                 &  \mbox{          }  &  \beta_j^w:  &  U  &  \longrightarrow  &  \bC  \\
                   &  z  &  \longmapsto      &  \alpha_i(z,w)  &  \mbox{          }  &                   &  z   &  \longrightarrow  &  \beta_j(z,w)
\end{matrix}$$

\vspace{0,3cm}

The equation (1) implies that $\sum\limits_i\alpha_i^{w}h_i=\sum\limits_j\beta_j^{w}g_j$, $\forall w \in U$. Applying $\phi$ (which is an $\cO_{X,x}-$homomorphism) in both sides of the last equation we get $\sum\limits_i\alpha_i^{w}\phi(h_i)=\sum\limits_j\beta_j^{w}\phi(g_j)$, $\forall w\in U$. This implies that $$(\sum\limits_i\alpha_i(\phi(h_i)\circ\pi_1))(z,w)=(\sum\limits_j\beta_j(\phi(g_j)\circ\pi_1))(z,w)$$ $\forall (z,w)\in U\times U$, hence $$\sum\limits_i\alpha_i(\phi(h_i)\circ\pi_1)=\sum\limits_j\beta_j(\phi(g_j)\circ\pi_1). \eqno (3)$$

Analogously, using the equation (2), we get $$\sum\limits_i\alpha_i(\phi(h_i)\circ\pi_2)=\sum\limits_j\beta_j(\phi(g_j)\circ\pi_2). \eqno (4)$$
The equations (3) and (4) implies that $$\sum\limits_i\alpha_i(\phi(h_i))_D=\sum\limits_j\beta_j(\phi((g_j))_D$$ and the Claim is proved.

Now, by the definition of $\phi_D$, it is clear that $\phi_D$ is an $\cO_{X\times X,(x,x)}-$module homomorphism and is the unique satisfying the property $\phi_D(h_D)=(\phi(h))_D$, $\forall h\in M$.
\end{proof}

\vspace{0,5cm}

From now on, all the modules are objects in $\mathcal T(\cO_{X,x})$ and their doubles are objects in $\mathcal T(\cO_{X\times X,(x,x)})$. 

\vspace{1cm}

Notice that if $id_M:M\rightarrow M$ and $id_{M_D}:M_D\rightarrow M_D$ are the identidy homomorphisms of $M$ and $M_D$, then $$(id_M)_D=id_{M_D}.$$

The next proposition gives us a relation between images and kernels of a module homomorphism.

\begin{proposition}\label{3.P4}
Let $\phi:M\rightarrow N$ be an $\cO_{X,x}-$module homomorphism and \\$\phi_D:M_D\rightarrow N_D$ its double. Then:
\begin{enumerate}
\item[a)] $Im(\phi_D)=(Im (\phi))_D$;

\vspace{0,2cm}

\item[b)] $(ker(\phi))_D\subset Ker(\phi_D)$.
\end{enumerate}
\end{proposition}

\begin{proof}
(a) By definition of $\phi_D$, it is clear that $Im(\phi_D)\subset (Im(\phi))_D$. Now, if $g\in Im(\phi)$ then we can write $g=\phi(h)$, for some $h\in M$. So, $g_D=(\phi(h))_D=\phi_D(h_D)\in Im(\phi_D)$. Thus, $g_D\in Im(\phi_D)$, $\forall g\in Im(\phi)$, hence $(Im(\phi))_D\subset Im(\phi_D)$.

\vspace{0,5cm}

(b) For every $h\in Ker(\phi)$, we have that $\phi_D(h_D)=(\phi(h))_D=(0_N)_D=0_{N_D}$, so $h_D\in Ker(\phi_D)$.
\end{proof}

The next proposition shows that the double homomorphism has a good behavior with respect to sum and composition.

\begin{proposition}\label{3.P9}
Let $\phi,\phi':M\rightarrow N$ and $\gamma:N\rightarrow P$ be $\cO_{X,x}-$module homomorphisms.
\begin{enumerate}

\item[a)] $\phi=\phi' \iff \phi_D=\phi'_D$;

\vspace{0,2cm}

\item[b)] $(\gamma\circ\phi)_D=\gamma_D\circ\phi_D$;

\vspace{0,2cm}

\item[c)] $(\phi+\phi')_D=\phi_D+\phi'_D$.

\end{enumerate}
\end{proposition}
\begin{proof}
(a) $(\Longrightarrow)$ Suppose $\phi=\phi'$. For all $h\in M$, $\phi_D(h_D)=(\phi(h))_D=(\phi'(h))_D=\phi'_D(h_D)$. Since the module $M_D$ is generated by $\{h_D\mbox{ / }h\in M\}$ then $\phi_D=\phi'_D$.

\vspace{0,2cm}

$(\Longleftarrow)$ Suppose $\phi_D=\phi'_D$. Given $h\in M$ arbitrary, we have $(\phi(h))_D=\phi_D(h_D)=\phi'_D(h_D)=(\phi'(h))_D$. By Proposition \ref{3.P1} (a), $\phi(h)=\phi'(h)$. Hence, $\phi=\phi'$.

\vspace{0,2cm}

(b) For every $h\in M$ we have $(\gamma\circ\phi)_D(h_D)=(\gamma\circ\phi(h))_D=\gamma_D((\phi(h))_D)=\gamma_D\circ\phi_D(h_D)$, which proves (b).

\vspace{0,2cm}

(c) For every $h\in M$ we have $(\phi+\phi')_D(h_D)=((\phi+\phi')(h))_D=(\phi(h)+\phi'(h))_D=(\phi(h))_D+(\phi'(h))_D=(\phi_D+\phi'_D)(h_D)$, which proves (c).

\end{proof}

\begin{corollary}\label{3.C5}
Let $\phi: M\rightarrow N$ be an $\cO_{X,x}-$module homomorphism. Then:
\begin{enumerate}
\item[a)] $\phi: M\rightarrow N$ is surjective if, and only if, $\phi_D:M_D\rightarrow N_D$ is a surjective;

\vspace{0,2cm}

\item[b)] If $\phi_D:M_D\rightarrow N_D$ is injective then $\phi: M\rightarrow N$ is injective;

\vspace{0,2cm}

\item[c)] $\phi:M\rightarrow N$ is an $\cO_{X,x}-$isomorphism if, and only if, $\phi_D: M_D\rightarrow N_D$ is an $\cO_{X\times X,(x,x)}-$isomorphism;

\vspace{0,2cm} 

\item[d)] $\phi: M\rightarrow N$ is the zero homomorphism if, and only if, $\phi_D:M_D\rightarrow N_D$ is the zero homomorphism.

\end{enumerate}
\end{corollary}

\begin{proof}
(a) By Propositions \ref{3.P1} (d) and \ref{3.P4} (a) we have: $\phi$ is surjective $\iff$ $Im(\phi)=N$ $\iff$  $(Im(\phi))_D=N_D$  $\iff$ $Im(\phi_D)=N_D$ $\iff$ $\phi_D$ is surjective. 

\vspace{0,5cm}

(b) If $\phi_D$ is injective then $Ker(\phi_D)=0_{M_D}$, and the Proposition \ref{3.P4} (b) implies that $(Ker(\phi))_D\subset Ker(\phi_D)=0_{M_D}=(0_M)_D$. By Proposition \ref{3.P1} (c) we conclude that $Ker(\phi)=0_M$, hence $\phi$ is injective.
\end{proof}

\vspace{0,5cm}

(c) $(\Longrightarrow)$ Since $\phi:M\rightarrow N$ is an isomorphism then there exists an $\cO_{X,x}-$homomorphism $\gamma:N\rightarrow M$ such that $\gamma\circ\phi=id_M$ and $\phi\circ\gamma=id_N$, then, by Proposition \ref{3.P9} (b) we have $(\gamma)_D\circ(\phi)_D=id_{M_D}$ and $(\phi)_D\circ(\gamma)_D=id_{N_D}$. Hence, $\phi_D$ is an isomorphism.

\vspace{0,5cm}

$(\Longleftarrow)$It follows immediately from (a) and (b).

\vspace{0,5cm}

(d) By Propositions \ref{3.P1} (d) and \ref{3.P4} (a) we have: $\phi$ is the zero homomorphism $\iff$ $Im(\phi)=0_N$ $\iff$  $(Im(\phi))_D=(0_N)_D$  $\iff$ $Im(\phi_D)=0_{N_D}$ $\iff$ $\phi_D$ is the zero homomorphism. 

\begin{definition}
We say that an $\cO_{X,x}-$homomorphism $\phi:M\subset\cO_{X,x}^p\rightarrow N\subset\cO_{X,x}^q$ is induced by a $q\times p$ matrix if there exists $A\in Mat_{q\times p}(\cO_{X,x})$ such that $\phi(h)=A\cdot h$, $\forall h\in M$.
\end{definition}

\begin{lemma}\label{3.L19}
An $\cO_{X,x}-$homomorphism $\phi:M\subset\cO_{X,x}^p\rightarrow N\subset\cO_{X,x}^q$ is induced by a $q\times p$ matrix if, and only if, there exists an $\cO_{X,x}-$homomorphism $\tilde{\phi}:\cO_{X,x}^p\rightarrow\cO_{X,x}^q$ such that $\tilde{\phi}(M)\subset N$ and $\tilde{\phi}\mid_{M}=\phi$.
\end{lemma}

\begin{proof}
$(\Longrightarrow)$ By hypothesis there exists a $q\times p$ matrix $A$ with entries in $\cO_{X,x}$ such that $\phi(h)=A\cdot h$, $\forall h\in M$. From this matrix $A$, we can define $\tilde{\phi}:\cO_{X,x}^p\rightarrow\cO_{X,x}^q$ given by $\tilde{\phi}(g):=A\cdot g$, which is an $\cO_{X,x}-$homomorphism. Clearly, $\tilde{\phi}\mid_{M}=\phi$, and for all $h\in M$ we have $\tilde{\phi}(h)=\phi(h)\in N$, so $\tilde{\phi}(M)\subset N$.

\vspace{0,5cm}

$(\Longleftarrow)$ Let $e_1,...,e_p$ be the canonical elements in $\cO_{X,x}^p$. Let $A$ be the $q\times p$ matrix whose columns are $\phi(e_1),...,\phi(e_p)$. Then $\tilde{\phi}(g)=A\cdot g,\forall g\in \cO_{X,x}^p$. Since $\tilde{\phi}\mid_{M}=\phi$ then $\phi(h)=\tilde{\phi}(h)=A\cdot h,\forall h\in M$. Therefore, $\phi$ is induced by a $q\times p$ matrix.
\end{proof}

\begin{proposition}\label{3.P20}
If $\phi:M\subset\cO_{X,x}^p\rightarrow N\subset\cO_{X,x}^q$ is an $\cO_{X,x}-$homomorphism induced by a $q\times p$ matrix then $$\phi_D:M_D\subset\cO_{X\times X,(x,x)}^{2p}\rightarrow N_D\subset\cO_{X\times X,(x,x)}^{2q}$$ is an $\cO_{X\times X,(x,x)}-$homomorphism induced by a $2q\times 2p$ matrix.
\end{proposition}

\begin{proof}
By hypothesis there exists a $q\times p$ matrix $A$ with entries in $\cO_{X,x}$ such that $\phi(h)=A\cdot h$, $\forall h\in M$. Then, for all $h\in M$ we have $\phi_D(h_D)=
\begin{bmatrix}
\phi(h)\circ\pi_1 \\
\phi(h)\circ\pi_2
\end{bmatrix}=
\begin{bmatrix}
(A\cdot h)\circ\pi_1 \\
(A\cdot h)\circ\pi_2
\end{bmatrix}
=
\begin{bmatrix}
(A\circ\pi_1)\cdot(h\circ\pi_1) \\
(A\circ\pi_2)\cdot(h\circ\pi_2)
\end{bmatrix}$.

\vspace{0,2cm}

So, taking the $2q\times 2p$ matrix 
$$B:= \begin{bmatrix}
A\circ\pi_1  &  0_{q\times p} \\
0_{q\times p} & A\circ\pi_2
\end{bmatrix}$$
we conclude that $\phi_D(h_D)=B\cdot h_D$, and the proposition is proved, once $M_D$ is generated by $h_D$, $h\in M$.
\end{proof}

As an application of the double homomorphism, we prove in the next theorem that the double structure is compatible with finite direct sum of modules.

\begin{theorem}\label{3.T21}
Let $M\subset\cO_{X,x}^p$ and $N\subset\cO_{X,x}^q$ be $\cO_{X,x}-$submodules. Then $$(M\oplus N)_D\cong M_D\oplus N_D$$ as $\cO_{X\times X,(x,x)}-$submodules of $\cO_{X\times X,(x,x)}^{2(p+q)}$. 

Furthermore, there exists an isomorphism $$\eta:(M\oplus N)_D\longrightarrow M_D\oplus N_D$$ such that $\eta((h,g)_D)=(h_D,g_D)$, for all $h\in M$ and $g\in N$.
\end{theorem}

\begin{proof}
Consider the canonical projections and inclusions:
$$\begin{matrix}
\psi_1:  &  M\oplus N  &  \longrightarrow  &  M   & \hspace{3cm}  &  \psi_2:  &  M\oplus N  &  \longrightarrow  &  N  \\
             &   (h,g)          &  \longmapsto      &  h    & \hspace{3cm}  &              &    (h,g)         &  \longmapsto      &  g   \\
             &                     &                           &                     &                            &                    &                                     &             &       \\
\delta_1:  &  M            &  \longrightarrow  &  M\oplus N   & \hspace{3cm}      &  \delta_2:    &  N               &  \longrightarrow  &  M\oplus N  \\
               &    h      &  \longmapsto       &  (h,0_N)                 & \hspace{3cm}      &               &    g         &  \longmapsto      &  (0_M,g)   \\
\end{matrix}$$
Thus, we get the double homomorphism of each one above:

$$\begin{matrix}
(\psi_1)_D:  &  (M\oplus N)_D  &  \longrightarrow  &  M_D   & \hspace{3cm}  &  (\psi_2)_D:  &  (M\oplus N)_D  &  \longrightarrow  &  N_D  \\
                          
\end{matrix}$$
$$\begin{matrix}                          
(\delta_1)_D:  &  M_D            &  \longrightarrow  &  (M\oplus N)_D   & \hspace{3cm}      &  (\delta_2)_D:    &  N_D               &  \longrightarrow  &  (M\oplus N)_D  \\               
\end{matrix}$$

Define:

$$\begin{matrix}
\eta:  &(M\oplus N)_D  &  \longrightarrow   &  M_D\oplus N_D \\
         &          w             &   \longmapsto      & ((\psi_1)_D(w),(\psi_2)_D(w))   
\end{matrix}$$

\vspace{0,3cm}

$$\begin{matrix}
\delta:  &  M_D\oplus N_D  &  \longrightarrow  &   (M\oplus N)_D  \\
            &      (u,v)                &  \longmapsto      &   (\delta_1)_D(u)+(\delta_2)_D(v) 
\end{matrix}$$ which are $\cO_{X\times X,(x,x)}-$module homomorphisms.

\vspace{0,3cm}

Claim 1: $\eta((h,g)_D)=(h_D,g_D)$, for all $h\in M$ and $g\in N$. 

In fact, $\eta((h,g)_D)=((\psi_1)_D((h,g)_D),(\psi_2)_D((h,g)_D))=((\psi_1(h,g))_D,(\psi_2(h,g))_D)=(h_D,g_D)$.

\vspace{0,4cm}

Claim 2: $\delta(h_D,g_D)=(h,g)_D$, for all $h\in M$ and $g\in N$.

In fact, $\delta(h_D,g_D)=(\delta_1)_D(h_D)+(\delta_2)_D(g_D)=(\delta_1(h))_D+(\delta_2(g))_D=(h,0_N)_D+(0_M,g)_D=((h,0_N)+(0_M,g))_D=(h,g)_D$.

\vspace{0,4cm}

By the Claims 1 and 2 we have that $\delta\circ\eta((h,g)_D)=(h,g)_D$, $\eta\circ\delta(h_D,0_{N_D})=(h_D,0_{N_D})$ and $\eta\circ\delta(0_{M_D},g_D)=(0_{M_D},g_D)$, for all $h\in M$ and $g\in N$. 

Since $\{(h,g)_D\mbox{ / } h\in M\mbox{ and }g\in N\}$ is a generator set of $(M\oplus N)_D$ and \\$\{(h_D,0_{N_D}),(0_{M_D},g_D)\mbox{ / } h\in M\mbox{ and }g\in N\}$ is a generator set of $M_D\oplus N_D$ then we conclude that $\delta\circ\eta=id_{(M\oplus N)_D}$ and $\eta\circ\delta=id_{M_D\oplus N_D}$, which finishes the proof of the theorem.
\end{proof}

\begin{corollary}\label{3.C22}
Let $M_i\subset\cO_{X,x}^{p_i}$ be $\cO_{X,x}-$submodules, for each $i\in\{1,...,r\}$. Then $$(M_1\oplus ...\oplus M_r)_D\cong(M_1)_D\oplus ... \oplus(M_r)_D$$ as $\cO_{X\times X,(x,x)}-$submodules of $\cO_{X\times X,(x,x)}^{2(p_1+...+p_r)}$ through an isomorphism such that $$(h_1,...,h_r)_D\longmapsto((h_1)_D,...,(h_r)_D)$$ for all $h_i\in M_i$.
\end{corollary}

\begin{proof}
Induction on $r$ and use the previous theorem.
\end{proof}

\begin{proposition}\label{P3.1.13}
	Let $M\subset N$ be $\cO_{X,x}-$submodules of $\cO_{X,x}^p$.
	\begin{enumerate}
		\item [a)] If $M_D$ has finite length then $M$ has finite length and $\ell(M)\leq\ell(M_D)$;
		
		\item[b)] If $M_D$ has finite colength in $N_D$ then $M$ has finite colength in $N$.
		
	\end{enumerate}
	
\end{proposition}

\begin{proof}
	(a) If $r\in\bN$ and $(M_i)_{i=0}^r$ is an ascending series of $M$, then $((M_i)_D)_{i=0}^r$ is an ascending series of $M_D$, which has finite length. Thus, $r\leq\ell(M_D)$. Therefore $\ell(M)$ is finite and $\ell(M)\leq\ell(M_D)$.
	
	(b) Let $r\in\bN$ and consider an arbitrary ascending series of $\frac{N}{M}$ of length $r$. This series can be given on the form $$\frac{N_0}{M}\subsetneq\frac{N_1}{M}\subsetneq ... \subsetneq \frac{N_{r-1}}{M}\subsetneq\frac{N_r}{M}=\frac{N}{M}$$ where $N_0\subsetneq N_1\subsetneq ...\subsetneq N_{r-1}\subsetneq N_r=N$ are $\cO_{X,x}$-submodules of $N$ which contain $M$. Then, $$\frac{(N_0)_D}{(M_D)}\subsetneq\frac{(N_1)_D}{M_D}\subsetneq ... \subsetneq \frac{(N_{r-1})_D}{M_D}\subsetneq\frac{(N_r)_D}{M_D}=\frac{N_D}{M_D}$$ is an ascending series of $\frac{N_D}{M_D}$, which has finite length by hypothesis. Hence, $r\leq\ell(\frac{N_D}{M_D})$ and $\ell(\frac{N}{M})$ is finite.  
	
\end{proof}

\section{HOMOLOGICAL ASPECTS OF THE DOUBLE STRUCTURE}\label{sec2}

\vspace{1cm}

\begin{proposition}\label{3.P6}
Let
$$\begin{matrix}
      &   \phi                   &       &   \gamma           &          \\
M   & \longrightarrow   &  N   & \longrightarrow &  P   
\end{matrix}$$ be an sequence of $\cO_{X,x}-$module homomorphism and consider the double sequence

$$\begin{matrix}
      &   \phi_D                   &       &   \gamma_D           &          \\
M_D   & \longrightarrow   &  N_D   & \longrightarrow &  P_D \mbox{                                    .}  
\end{matrix}$$ 

\vspace{0,3cm}

If $Im(\phi)\subset Ker(\gamma)$ then $Im(\phi_D)\subset Ker(\gamma_D)$.
\end{proposition}

\begin{proof}
Since $Im(\phi)\subset Ker(\gamma)$ then $(Im(\phi))_D\subset(Ker(\gamma))_D$. Hence, $Im(\phi_D)=(Im(\phi))_D\subset(Ker(\gamma))_D\subset Ker(\gamma_D)$.
\end{proof}

We will see that the double homomorphism gives a natural way to study the homology of the double structure.

\begin{definition}[The double chain complex]\label{3.D7}
	Let $\cC=(M_{\bullet},\phi_{\bullet})$ be a chain complex in $\mathcal T(\cO_{X,x})$. We define $$\cC_D:=((M_{\bullet})_D,(\phi_{\bullet})_D)$$ and by Proposition \ref{3.P6} we have that $\cC_D$ is a chain complex in $\mathcal T(\cO_{X\times X,(x,x)})$. The chain complex $\cC_D$ is called {\bf the double of $\cC$}.
\end{definition}

\begin{proposition}\label{3.P8}
Let $C=(M_{\bullet},\phi_{\bullet})$ be a chain complex. If $C_D$ is an exact sequence then $\cC$ is an exact sequence. In other words, if $\cC_D$ has trivial homology then $\cC$ has trivial homology.
\end{proposition}

\begin{proof}
Let $i\in \mathbb Z$ be arbitrary. We have the sequences 

$$\begin{matrix}
                &   \phi_{i+1}         &           &   \phi_i              &               &                &(\phi_{i+1})_D            &            &   (\phi_i)_D             &          \\
M_{i+1}   & \longrightarrow   &  M_i   & \longrightarrow &  M_{i-1} \hspace{2cm} & (M_{i+1})_D   & \longrightarrow   &  (M_i)_D   & \longrightarrow &  (M_{i-1})_D   
\end{matrix}$$ 

\vspace{0,3cm}

We already know that $Im(\phi_{i+1})\subset Ker(\phi_i)$. Since $\cC_D$ is an exact sequence then $Im((\phi_{i+1})_D)=Ker((\phi_i)_D)$. By Proposition \ref{3.P4} we have $(Ker(\phi_i))_D\subset Ker((\phi_i)_D)=Im((\phi_{i+1})_D)=(Im(\phi_{i+1}))_D$. By Proposition \ref{3.P1} (c), we conclude that $Ker(\phi_i)\subset Im(\phi_{i+1})$. Therefore, $Im(\phi_{i+1})=Ker(\phi_i)$.

\end{proof}

\begin{proposition}\label{3.P10}
Let $\cC=(M_{\bullet},\phi_{\bullet})$ and $\cC'=(M'_{\bullet},\phi'_{\bullet})$ be chain complexes.\\ If $\alpha:\cC\longrightarrow\cC'$ is a chain complex morphism then $\alpha_D:\cC_D\longrightarrow\cC'_D$ given by \\$\{(\alpha_i)_D:(M_i)_D\rightarrow(M'_i)_D\mbox{ / } i\in\mathbb{Z}\}$ is a chain complex morphism, called the {\bf the double morphism of} $\alpha$.
\end{proposition}

\begin{proof}
Let $i\in\mathbb{Z}$. So we have that the diagram 
$$\begin{matrix}
       &                           &  \phi_i                  &                           &           \\
       & M_i                       &  \longrightarrow  & M_{i-1}                  &            \\
\alpha_i & \downarrow &                           & \downarrow & \alpha_{i-1} \\
       & M'_i                   &  \longrightarrow  & M'_{i-1}                   &              \\     
       &                           &  \phi'_i                  &                           &                                        
\end{matrix}$$ is commutative. By Proposition \ref{3.P9} (b) follows that $(\phi'_i)_D\circ(\alpha_i)_D=(\phi'_i\circ\alpha_i)_D=(\alpha_{i-1}\circ\phi_i)_D=(\alpha_{i-1})_D\circ(\phi_i)_D$, and the following diagram is also commutative:
$$\begin{matrix}
       &                           &  (\phi_i)_D                  &                           &           \\
       & (M_i)_D                       &  \longrightarrow  & (M_{i-1})_D                  &            \\
(\alpha_i)_D & \downarrow &                           & \downarrow & (\alpha_{i-1})_D \\
       & (M'_i)_D                   &  \longrightarrow  & (M'_{i-1})_D                   &              \\     
       &                           &  (\phi'_i)_D                  &                           &                                        
\end{matrix}$$
\end{proof}

\begin{corollary}\label{3.C11}
If $\alpha:\cC\longrightarrow\cC'$ and $\beta:\cC'\longrightarrow\cC''$ are chain morphisms then $$(\beta\circ\alpha)_D=\beta_D\circ\alpha_D.$$
\end{corollary}
\begin{proof}
It is a straightforward consequence of the Proposition \ref{3.P9} (b).
\end{proof}

\vspace{0,5cm}

Now, we will get some results related to chain homotopy. 

Let $\cC=(M_{\bullet},\phi_{\bullet})$ and $\cC'=(M'_{\bullet},\phi'_{\bullet})$ be chain complexes. Let $\mu:\cC\rightarrow\cC'$ be a homomorphism of degree 1, i.e, $\mu$ is a collection of $\cO_{X,x}-$module homomorphisms $\{\mu_i:M_i\rightarrow M'_{i+1}\mbox{ / }i\in\mathbb{Z}\}$. We know this homomorphism induces a chain morphism $\tilde{\mu}:\cC\rightarrow \cC'$ given by $\{\tilde{\mu_i}:M_i\rightarrow M'_i\mbox{ / }i\in\mathbb{Z}\}$, where $\tilde{\mu_i}:=\phi'_{i+1}\circ\mu_i+\mu_{i-1}\circ\phi_i,\forall i\in\mathbb{Z}$.

If $\alpha,\beta:\cC\rightarrow\cC'$ are chain morphisms, remember that $\mu:\cC\rightarrow\cC'$ is defined as a {\bf homotopy between $\alpha$ and $\beta$} when $\tilde{\mu}=\alpha-\beta$, and we denote $\alpha\simeq\beta$ by $\mu$.

\vspace{0,2cm}

\begin{lemma}\label{3.L12}
Consider $\mu_D:\cC_D\rightarrow\cC'_D$ the homomorphism of degree 1 given by the double homomorphisms of $\mu:\cC\rightarrow\cC'$. Then $\widetilde{\mu_D}=(\tilde{\mu})_D$.
\end{lemma}

\begin{proof}
For all $i\in\mathbb{Z}$ we have $(\widetilde{\mu_D})_i=(\phi'_{i+1})_D\circ(\mu_i)_D+(\mu_{i-1})_D\circ(\phi_i)_D\\=(\phi'_{i+1}\circ\mu_i+\mu_{i-1}\circ\phi_i)_D=(\tilde{\mu_i})_D$, and the lemma is proved.
\end{proof}

\vspace{0,2cm}

\begin{proposition}\label{3.P13}
Let $\alpha,\beta:\cC\rightarrow\cC'$ be chain morphisms and $\mu:\cC\rightarrow\cC'$ a homomorphism of degree 1.
Then: $\mu$ is a homotopy between $\alpha$ and $\beta$ if, and only if, $\mu_D$ is a homotopy between $\alpha_D$ and $\beta_D$. 
\end{proposition}

\begin{proof}
We have that $\mu$ is a homotopy between $\alpha$ and $\beta$ $\iff \tilde{\mu}=\alpha-\beta$. By Proposition \ref{3.P9} (a) and (c) and the previous lemma we have: $\tilde{\mu}=\alpha-\beta \iff (\tilde{\mu})_D=(\alpha-\beta)_D \iff \widetilde{\mu_D}=\alpha_D-\beta_D \iff$ $\mu_D$ is a homotopy between $\alpha_D$ and $\beta_D$.
\end{proof}

\vspace{0,2cm}

\begin{corollary}\label{3.C14}
If $\alpha:\cC\rightarrow\cC'$ is a chain homotopy equivalence then $\alpha_D:\cC_D\rightarrow\cC'_D$ is a chain homotopy equivalence.
\end{corollary}

\begin{proof}
By hypothesis there exists a chain morphism $\beta:\cC'\rightarrow\cC$ such that $\beta\circ\alpha\simeq id_{\cC}$ and $\alpha\circ\beta\simeq id_{\cC'}$. By the previous proposition, we have $(\beta\circ\alpha)_D\simeq (id_{\cC})_D$ and $(\alpha\circ\beta)_D\simeq (id_{\cC'})_D$, and therefore $(\beta)_D\circ(\alpha)_D\simeq id_{\cC_D}$ and $\alpha_D\circ\beta_D\simeq id_{\cC'_D}$.
\end{proof}

\begin{corollary}\label{3.C15}
If $\cC$ is a contractible chain complex then $\cC_D$ is also contractible.
\end{corollary}
\begin{proof}
Since $\cC$ is contractible then $id_{\cC}\simeq 0_{\cC}$, and by the Proposition \ref{3.P13} follows that $(id_{\cC})_D\simeq (0_{\cC})_D$, thus $id_{\cC_D}\simeq 0_{\cC_D}$. Hence, $\cC_D$ is contractible. 
\end{proof}

Notice the nice relation between the Proposition \ref{3.P8} and Corollary \ref{3.C15}. We already know every contractible chain complex is an exact sequence. The Proposition \ref{3.P8} says that the exactness on the double level implies the exactness on the single level. The Corollary \ref{3.C15}, which treats about contractible (stronger than exactness), says the opposite.

\vspace{0,2cm}

It is clear that all the results obtained in this section can be naturally translated to the cohomology language.

\vspace{0,5cm}
\section{THE DOUBLE CATEGORY}\label{sec3}

Let us define the category $\cD(\cO_{X,x})$. 

The objects of $\cD(\cO_{X,x})$ consist of the double of modules in $\mathcal T(\cO_{X,x})$. Given $M_D,N_D$ objects in $\cD(\cO_{X,x})$, we define 
$$Mor(M_D,N_D):=\{\phi_D:M_D\rightarrow N_D\mbox{ / }\phi:M\rightarrow N\mbox{ is an } \cO_{X,x}-\mbox{module homomorphism}\}.$$

Working with the standard composition of maps, we have the category $\cD(\cO_{X,x})$, called the {\bf double category of} $(X,x)$.

The morphisms in $\cD(\cO_{X,x})$ are called $\cO_{X,x}-$double morphisms. Observe that the $\cO_{X,x}-$doubles morphisms are $\cO_{X\times X,(x,x)}-$homomorphisms with an addictional property: they preserve the double structure.

Notice that $\cD(\cO_{X,x})$ is a subcategory of $\mathcal T(\cO_{X\times X,(x,x)})$.

\begin{theorem}\label{3.T16}
The covariant functor
$$\begin{matrix}
\mathbf{D}:   &   \mathcal T(\cO_{X,x})  &  \longrightarrow  & \cD(\cO_{X,x})  \\
                     &                M                    &   \longmapsto    &      M_D            \\
                     &       \phi:M\rightarrow N  &   \longmapsto    & \phi_D:M_D\rightarrow N_D
\end{matrix}$$ is an isomorphism of categories.
\end{theorem}

\begin{proof}
The Proposition \ref{3.P1} (d) proves that the map between the objects is a bijection, and the Proposition \ref{3.P9} (a) proves that the map between the morphisms is a bijection. Hence, $\mathbf{D}$ is an isomorphism of categories.
\end{proof}

\begin{corollary}\label{3.C17}
$\mathcal T(\cO_{X,x})$ can be seen as a subcategory of $\mathcal T(\cO_{X\times X,(x,x)})$.
\end{corollary}

The Theorem \ref{3.T16} implies that $\mathcal T(\cO_{X,x})$ and $\cD(\cO_{X,x})$ are essentially the same category, so they have the same behavior in all of the categorical statements. But, one of them is reasonable to emphasize, in the next corollary, which is interesting to compare with the result obtained in the Corollary \ref{3.C5}.

\begin{corollary}\label{3.C18}
Let $\phi:M\rightarrow N$ be an $\cO_{X,x}-$module homomorphism. Then:
\begin{enumerate}
\item[a)] $\phi:M\rightarrow N$ is an $\cO_{X,x}-$monomorphism of modules\\ if, and only if, $\phi_D:M_D\rightarrow N_D$ is an $\cO_{X,x}-$double monomorphism;

\vspace{0,2cm}

\item[b)] $\phi:M\rightarrow N$ is an $\cO_{X,x}-$epimorphism of modules if, and only if, $\phi_D:M_D\rightarrow N_D$ is an $\cO_{X,x}-$double epimorphism;

\vspace{0,2cm}

\item[c)] $\phi:M\rightarrow N$ is an $\cO_{X,x}-$isomorphism of modules if, and only if, $\phi_D:M_D\rightarrow N_D$ is an $\cO_{X,x}-$double isomorphism.
\end{enumerate}
\end{corollary} 

\vspace{0,5cm}

Here it is reasonable to emphasize the difference between the notions of injective homomorphism and monomorphism of modules, which are not the samething in the category of modules, since we are understanding the term {\bf monomorphism} in the categorical sense, i.e, there is a left-inverse morphism.

The same remark has to be done between surjective homomorphism and epimorphism.

\begin{remark}
The covariant functor
$$\begin{matrix}
\mathbf{D}:   &   \mathcal T(\cO_{X,x})  &  \longrightarrow  & \mathcal T(\cO_{X\times X,(x,x)})  \\
                     &                M                    &   \longmapsto    &      M_D            \\
                     &       \phi:M\rightarrow N  &   \longmapsto    & \phi_D:M_D\rightarrow N_D
\end{matrix}$$ is not an isomorphism of categories anymore. 
\end{remark}

In fact, suppose $(X,x)$ irreducible. It is proved in \cite{GT} that the generic rank of the double of every module $M$ in $\mathcal T(\cO_{X,x})$ has generic rank even. Thus, the map between the objects cannot be surjective.

\vspace{0,5cm}

\section{THE DOUBLE IN A QUOTIENT OF A FREE $\cO_X$-MODULE OF FINITE RANK}

Let $W$ be an $\cO_{X,x}-$submodule of $\cO_{X,x}^p$, and consider the quotient map $$\pi:\cO_{X,x}^p\longrightarrow \frac{\cO_{X,x}^p}{W}\hspace{2cm}.$$

\vspace{0,4cm}

Then, $W_D$ is an $\cO_{X\times X,(x,x)}-$submodule of $\cO_{X\times X,(x,x)}^{2p}$.

\vspace{0,3cm}

Let $h\in\cO_{X,x}^p$. We define the double of $h+W\in\frac{\cO_{X,x}^p}{W}$ as $$(h+W)_D:=h_D+W_D\in\frac{\cO_{X\times X,(x,x)}^{2p}}{W_D}.$$

Notice the definition of $(h+W)_D$ does not depend of the choice of the representative $h$. In fact, if $h+W=g+W$ then $h-g\in W\Longrightarrow (h-g)_D\in W_D\\ \Longrightarrow h_D-g_D\in W_D\Longrightarrow h_D+W_D=g_D+W_D$.

\vspace{0,3cm}

Now, we want to define the double of a submodule M of $\frac{\cO_{X,x}^p}{W}$. We have that $\pi^{-1}(M)$ is a submodule of $\cO_{X,x}^p$ and $\pi^{-1}(M)\supset W$, hence $(\pi^{-1}(M))_D$ is a submodule of $\cO_{X\times X,(x,x)}^{2p}$ and $(\pi^{-1}(M))_D\supset W_D$.

Then, we define the double of $M$ as $$M_D:=\frac{(\pi^{-1}(M))_D}{W_D}$$ which is an $\cO_{X\times X,(x,x)}-$submodule of $\frac{\cO_{X\times X,(x,x)}^{2p}}{W_D}$. 

\vspace{0,3cm}

If we call $\tilde{M}:=\pi^{-1}(M)$, then $$M=\frac{\tilde{M}}{W}\mbox{ and }M_D=\frac{\tilde{M}_D}{W_D}.$$

Rewriting with standard notation, we conclude that, if $M$ is a submodule of $\cO_{X,x}^p$ and $M\supset W$ then $$\left(\frac{M}{W} \right)_D=\frac{M_D}{W_D}$$ and is generated by $\{(h+W)_D\mbox{ / }h\in M\}$.

\section{The double homomorphism relative to an analytic map germ}

Let $(Y,y)$ and $(X,x)$ be germs of analytic spaces, and let $\varphi:(Y,y)\rightarrow(X,x)$ be an analytic map germ. So, the pullback map $\varphi^*:\cO_{X,x}\rightarrow\cO_{Y,y}$ is a ring homomorphism, which induces an $\cO_{X,x}-$algebra structure in $\cO_{Y,y}$. Thus, every $\cO_{Y,y}-$module is also an $\cO_{X,x}-$module through this ring homomorphism.

We will see that there is a natural $\cO_{X\times X,(x,x)}-$algebra structure in $\cO_{Y\times Y,(y,y)}$ induced by the pullback of $\varphi$. In fact, let $$\mu_{X,x}:\cO_{X,x}\otimes_{\bC}\cO_{X,x}\longrightarrow \cO_{X\times X,(x,x)}$$ be the $\bC-$algebra homomorphism such that $\mu_{X,x}(f\otimes_{\bC}g)$ is the germ of the map 

$$\begin{matrix}
U\times U  &  \rightarrow   &  \bC \\
(u,v)          &  \mapsto       & f(u).g(v)
\end{matrix}$$ and let $$\mu_{Y,y}:\cO_{Y,y}\otimes_{\bC}\cO_{Y,y}\longrightarrow \cO_{Y\times Y,(y,y)}$$ the same for $(Y,y)$. 

Since $\varphi^*:\cO_{X,x}\rightarrow\cO_{Y,y}$ is a ring homomorphism then we have a natural $\bC-$algebra homomorphism $$\varphi^{\otimes}: \cO_{X,x}\otimes_{\bC}\cO_{X,x}\longrightarrow \cO_{Y,y}\otimes_{\bC}\cO_{Y,y}$$ such that $\varphi^{\otimes}(f\otimes_{\bC}g)=(\varphi^*(f))\otimes_{\bC}(\varphi^*(g))$, $\forall f,g\in\cO_{X,x}$. In fact, the map
$$\begin{matrix}
\cO_{X,x}\times \cO_{X,x}  &  \longrightarrow  & \cO_{Y,y}\otimes_{\bC} \cO_{Y,y} \\
(f,g)                                     & \longmapsto       & (\varphi^*(f))\otimes_{\bC}(\varphi^*(g))       
\end{matrix}$$ is $\bC-$bilinear. So, the existence and uniqueness of $\varphi^{\otimes}$ is provided by the universal property of the tensor product. It is known that $\mu_{X,x}$ and $\mu_{Y,y}$ are $\bC-$algebra isomorphisms, so we can consider the $\bC-$algebra homomorphism $\epsilon_{\varphi}:\cO_{X\times X,(x,x)}\rightarrow \cO_{Y\times Y,(y,y)}$ such that the following diagram is commutative:

$$\begin{matrix}
                             &                                                          &         \mu_{X,x}          &                                      &                \\
                             &  \cO_{X,x}\otimes_{\bC}\cO_{X,x}    &     \longrightarrow      & \cO_{X\times X,(x,x)}   &                \\
\varphi^{\otimes}  & \downarrow                                       &                                   & \downarrow                  & \epsilon_{\varphi} \\        
                             &  \cO_{Y,y}\otimes_{\bC}\cO_{Y,y}    &     \longrightarrow      & \cO_{Y\times Y,(y,y)}     &                \\
                             &                                                          &         \mu_{Y,y}          &                                      &                
\end{matrix}$$

Since $\mu_{X,x}$ and $\mu_{Y,y}$ are $\bC-$algebra isomorphisms then we can identify $\epsilon_{\varphi}\cong\varphi^{\otimes}$, and $\varphi^{\otimes}:\cO_{X\times X,(x,x)}\rightarrow \cO_{Y\times Y,(y,y)}$ induces in $\cO_{Y\times Y,(y,y)}$ an $\cO_{X\times X,(x,x)}-$algebra structure.  

\begin{lemma}\label{3.L23}
Let $\alpha\in\cO_{X\times X,(x,x)}$. Suppose that $U$ is an open subset of $X$ containing $x$ where a representative of $\alpha$ is defined on $U\times U$. For each $w\in U$ let $\alpha^w\in\cO_{X,x}$ be the germ of the map 
$$\begin{matrix}
\alpha^w:  &U & \rightarrow  &  \bC \\
                & z & \mapsto  &  \alpha(z,w)
\end{matrix}$$

\noindent For each $y_2\in \varphi^{-1}(U)$ let $(\varphi^{\otimes}(\alpha))^{y_2}\in\cO_{Y,y}$ be the germ of the map 
$$\begin{matrix}
(\varphi^{\otimes}(\alpha))^{y_2}:  & \varphi^{-1}(U) & \rightarrow  &  \bC \\
                                                      & y_1 & \mapsto  &  (\varphi^{\otimes}(\alpha))(y_1,y_2)

\end{matrix}$$

Then $$\varphi^*(\alpha^{\varphi(y_2)})=(\varphi^{\otimes}(\alpha))^{y_2}, \forall y_2\in\varphi^{-1}(U).$$
\end{lemma}

\begin{proof}
We can write $\alpha=\sum(f_i\otimes_{\bC}g_i)$, with $f_i,g_i\in\cO_{X,x}$. For all $y_1\in\varphi^{-1}(U)$ we have:

$\varphi^*(\alpha^{\varphi(y_2)})(y_1)=\alpha^{\varphi(y_2)}(\varphi(y_1))=\alpha(\varphi(y_1),\varphi(y_2))=\sum(f_i(\varphi(y_1))\otimes_{\bC}g_i(\varphi(y_2)))=\left(\sum(\varphi^*(f_i))\otimes_{\bC}((\varphi^*(g_i))\right)(y_1,y_2)=(\varphi^{\otimes}(\alpha))(y_1,y_2)=(\varphi^{\otimes}(\alpha))^{y_2}(y_1)$, and the lemma is proved.
\end{proof}

Clearly we get the analogous result if we fix the first coordinate instead the second one. 

\vspace{0,5cm}

Consider the projections $\pi_1^{X},\pi_2^{X}:X\times X\rightarrow X$ and $\pi_1^{Y},\pi_2^{Y}:Y\times Y\rightarrow Y$.

\vspace{0,5cm}

\begin{theorem}\label{3.T24}
Let $M\subset\cO_{X,x}^{p}$ and $N\subset\cO_{Y,y}^q$ be submodules. If $\phi: M\rightarrow N$ is an $\cO_{X,x}-$module homomorphism then there exists a unique $\cO_{X\times X,(x,x)}-$module homomorphism $\phi_{D,\varphi}=\phi_D: M_D\rightarrow N_D$ such that $\phi_D(h_D)=(\phi(h))_D$, $\forall h\in M$.

The map $\phi_{D,\varphi}=\phi_D$ is called the double of $\phi$ relative to $\varphi:(Y,y)\rightarrow(X,x)$.
\end{theorem}

\begin{proof}
Since $M_D$ is generated by $\{h_D \mbox{ / }h\in M\}$ then we can define $\phi_D: M_D\rightarrow N_D$ in a natural way: for each $u=\sum\limits_i\alpha_i(h_i)_D$ with $\alpha_i\in\cO_{X\times X,(x,x)}$ and $h_i\in M$ we define $$\phi_D(u):=\sum\limits_i\alpha_i(\phi(h_i))_D=\sum\limits_i\varphi^{\otimes}(\alpha_i)(\phi(h_i))_D$$ which belongs to $N_D$.

{\bf Claim : $\phi_D$ is well defined.}
In fact, suppose that $\sum\limits_i\alpha_i(h_i)_D=\sum\limits_j\beta_j(g_j)_D$, with $\alpha_i,\beta_j\in\cO_{X\times X,(x,x)}$ and $h_i,g_j\in M$. So, we get two equations:

$$\sum\limits_i\alpha_i(h_i\circ\pi_1^{X})=\sum\limits_j\beta_j(g_j\circ\pi_1^{X}) \eqno (1)$$

$$\sum\limits_i\alpha_i(h_i\circ\pi_2^{X})=\sum\limits_j\beta_j(g_j\circ\pi_2^{X}). \eqno (2)$$

Take $U$ an open neighborhood of $x$ in $X$ where $\alpha_i,\beta_j$ are defined on $U\times U$, and $h_i,g_j$ are defined on $U$. For each $w\in U$ define $\alpha_i^w,\beta_j^w\in\cO_{X,x}$ given by the germs of the maps 
$$\begin{matrix}
\alpha_i^w: & U  &  \longrightarrow  &  \bC                 &  \mbox{          }  &  \beta_j^w:  &  U  &  \longrightarrow  &  \bC  \\
                   &  z  &  \longmapsto      &  \alpha_i(z,w)  &  \mbox{          }  &                   &  z   &  \longrightarrow  &  \beta_j(z,w)
\end{matrix}$$

\vspace{0,3cm}

The equation (1) implies that $\sum\limits_i\alpha_i^{w}h_i=\sum\limits_j\beta_j^{w}g_j$, $\forall w \in U$. Applying $\phi$ (which is a $\cO_{X,x}-$homomorphism) in both sides of the last equation we get $$\sum\limits_i\alpha_i^{w}\phi(h_i)=\sum\limits_j\beta_j^{w}\phi(g_j), \forall w\in U.$$

By the $\cO_{X,x}-$module structure on $N$ induced by $\varphi^*$, the last equation boils down to $$\sum\limits_i\varphi^*(\alpha_i^{w})\phi(h_i)=\sum\limits_j\varphi^*(\beta_j^{w})\phi(g_j), \forall w\in U.$$

By Lemma \ref{3.L23} we conclude that $$\sum\limits_i(\varphi^{\otimes}(\alpha_i))^{y_2}\phi(h_i)=\sum\limits_j(\varphi^{\otimes}(\beta_j))^{y_2}\phi(g_j), \forall y_2\in \varphi^{-1}(U).$$

Hence, $$\sum\limits_i\varphi^{\otimes}(\alpha_i)(\phi(h_i)\circ\pi_1^{Y})=\sum\limits_j\varphi^{\otimes}(\beta_j)(\phi(g_j)\circ\pi_1^{Y}).$$

Working with the analogous result of the Lemma \ref{3.L23}, the equation (2) implies that

$$\sum\limits_i\varphi^{\otimes}(\alpha_i)(\phi(h_i)\circ\pi_2^{Y})=\sum\limits_j\varphi^{\otimes}(\beta_j)(\phi(g_j)\circ\pi_2^{Y}).$$

Therefore, $$\sum\limits_i\varphi^{\otimes}(\alpha_i)(\phi(h_i))_D=\sum\limits_j\varphi^{\otimes}(\beta_j)(\phi(g_j))_D$$ and $\phi_D$ is well-defined.

Now, by the definition of $\phi_D$, it is clear that $\phi_D$ is an $\cO_{X\times X,(x,x)}-$module homomorphism and is the unique satisfying the property $\phi_D(h_D)=(\phi(h))_D$, $\forall h\in M$, i.e, $$\phi_D(h\circ\pi_1^{X},h\circ\pi_2^{X})=(\phi(h)\circ\pi_1^{Y},\phi(h)\circ\pi_2^{Y}).$$

\end{proof}

Notice that this approach generalizes what we have defined in Section \ref{sec1}, taking $\varphi:(X,x)\rightarrow(X,x)$ as the identidy map. The main motivation of this approach is the fact that when we work with integral closure of modules, the analytic curves $\varphi:(\bC,0)\rightarrow(X,x)$ has a key role.

Clearly the Propositions \ref{3.P4}, \ref{3.P9} (a,c) and the Corollary \ref{3.C5} (a,b,d) still hold for the double homomorphism relative to an analytic map.

We can write the Proposition \ref{3.P9} (b) on this new language as follows:

\begin{proposition}\label{3.P25}
Let $\varphi:(Y,y)\rightarrow(X,x)$ and $\varphi':(Z,z)\rightarrow(Y,y)$ be analytic map germs, $M\subset\cO_{X,x}^{p}$, $N\subset\cO_{Y,y}^{q}$ and $P\subset\cO_{Z,z}^{r}$ submodules. Let $\phi:M\rightarrow N$ be an $\cO_{X,x}-$module homomorphism and $\phi':N\rightarrow P$ be an $\cO_{Y,y}-$module homomorphism. Then, $\phi'\circ\phi:M\rightarrow P$ is an $\cO_{X,x}-$module homomorphism, considering $P$ with the $\cO_{X,x}-$module structure induced by the pullback of $\varphi\circ\varphi':(Z,z)\rightarrow(X,x)$ and $$(\phi'\circ\phi)_{D,\varphi\circ\varphi'}=\phi'_{D,\varphi'}\circ\phi_{D,\varphi}.$$

\end{proposition}

\begin{proof}
For all $\alpha\in\cO_{X,x}$ and $h\in M$, working with the module structures induced by the pullbacks of the analytic map germs, we have:

$\phi'\circ\phi(\alpha h)=\phi'(\alpha\phi(h))=\phi'(\varphi^*(\alpha)\phi(h))=\varphi^*(\alpha)\phi'(\phi(h))=\varphi'^*(\varphi^*(\alpha))(\phi'\circ\phi(h))=(\varphi\circ\varphi')^*(\alpha)(\phi'\circ\phi(h))=\alpha(\phi'\circ\phi(h))$. So $\phi'\circ\phi:M\rightarrow P$ is an $\cO_{X,x}-$module homomorphism and $(\phi'\circ\phi)_{D,\varphi\circ\varphi'}$ is well defined and clearly is equal to $\phi'_{D,\varphi'}\circ\phi_{D,\varphi}$.
\end{proof}

\vspace{2cm}

\begin{small}
	
	{\sc Thiago Filipe da Silva
		
		Instituto de Ci\^encias Matem\'aticas e de Computa\'c\~ao - USP \\
		Av. Trabalhador S\~ao Carlense, 400 - Centro, 13566-590 - S\~ao Carlos - SP, Brazil, thiago.filipe@usp.br
		
		\vspace{0,5cm}
		
		Departamento de Matem\'atica, Universidade Federal do Esp\'irito Santo \\
		Av. Fernando Ferrari, 514 - Goiabeiras, 29075-910 - Vit\'oria - ES, Brazil, thiago.silva@ufes.br}
	
	\vspace{1cm}

\end{small}

\end{document}